\documentclass[12pt,a4paper]{amsart}%

\usepackage{amsmath}
\usepackage{fullpage}
\usepackage{MyCommA}
\usepackage[final]{pdfpages}

\newcommand{\Eitan}[1]{{{#1}}}
\newcommand{\EitanB}[1]{{{#1}}}
\newcommand{\EitanC}[1]{{{#1}}}
\newcommand{\EitanD}[1]{{{#1}}}
\newcommand{\EitanE}[1]{{{#1}}}
\newcommand{\EitanF}[1]{{{#1}}}

\newcommand{\Rami}[1]{{{#1}}}

\newcommand{\RamiE}[1]{{{#1}}}

\newcommand{\RamiG}[1]{{{#1}}}
\newcommand{\RamiH}[1]{{{#1}}}
\newcommand{\RamiI}[1]{{{#1}}}

\newcommand{\RamiJ}[1]{{#1}}
\newcommand{\RamiK}[1]{{#1}}
\newcommand{\RamiL}[1]{{#1}}
\newcommand{\RamiM}[1]{{{#1}}}
\newcommand{\RamiN}[1]{{{#1}}}
\newcommand{\RamiO}[1]{{{#1}}}

\newcommand{\NextVer}[1]{}

\newcommand{\spa}{\mathrm{span}}

\newcommand{\CM}{Cohen-Macaulay} %hyperbolic distance
\begin{document}

\author[Aizenbud]{Avraham Aizenbud}
\address{Avraham Aizenbud,
Faculty of Mathematical Sciences,
Weizmann Institute of Science,
76100
Rehovot, Israel}
\email{aizenr@gmail.com}
\urladdr{http://aizenbud.org}

\author[Bernstein]{Joseph Bernstein}
\address{Joseph Bernstein,
		School of Mathematical Sciences, Tel Aviv University, Tel Aviv 6997801, Israel}
\email{bernstei@tauex.tau.ac.il}
\urladdr{http://www.math.tau.ac.il/~bernstei/}

\author[Sayag]{Eitan Sayag}
\address{Eitan Sayag,
 Department of Mathematics,
Ben Gurion University of the Negev,
P.O.B. 653,
Be'er Sheva 84105,
ISRAEL}
 \email{eitan.sayag@gmail.com}

\urladdr{https://www.math.bgu.ac.il/~sayage/}

\date{\today}
%
%20-xx  Group theory and generalizations 20Cxx Representation
%theory of groups 20C99 None of the above, but in this section
%
%20Gxx Linear algebraic groups (classical groups) 20G05
%Representation theory 20G25 Linear algebraic groups over local
%fields and their integers
%
%
%22-xx  Topological groups, Lie groups 22Exx Lie groups 22E45
%Representations of Lie and linear algebraic groups over real
%fields: analytic methods {For the purely algebraic theory, see
%20G05}
%
%
%46-xx  Functional analysis 46Fxx Distributions, generalized
%functions, distribution spaces 46F10 Operations with distributions
%
%14-xx  Algebraic geometry 14Lxx Algebraic Groups 14L24 Geometric
%invariant theory [See also 13A50] 14L30 Group actions on varieties
%or schemes (quotients) [See also 13A50, 14L24]
%
%43A85  	Harmonic analysis on homogeneous spaces

%\date{\today}
\title{%Cohen-Macaulay 
%property of the category of smooth representations of reductive p-adic groups,  projectivity of the small Whittaker model, and strong density of Whittaker spherical characters
%and 
Strong density of spherical characters attached to unipotent subgroups}
\maketitle
\begin{abstract}
We prove the following \EitanD{result in relative representation theory of a reductive  p-adic group $G$:}

\RamiJ{Let $U$ be the unipotent radical of a minimal parabolic subgroup of $G$, and let $\psi$ be an arbitrary smooth character of $U$.}
Let $S \subset \Irr(G)$ be a Zariski dense collection of irreducible representations of $G$. Then the span of the Bessel distributions $B_{\pi}$ attached to representations $\pi$ from $S$ is dense in the space  $\Sc^*(G)^{U\times U,\psi \times \psi}$ of all $(U\times U,\psi \times \psi)$-equivariant distributions on $G.$

\EitanD{We base our proof on the following results.}
\begin{enumerate}
	\item The category of smooth representations $\cM(G)$ is \CM.
	\item \RamiJ{The module} \RamiH{$ind_U^G(\psi)$} is a projective module.
\end{enumerate}
\end{abstract}

\tableofcontents

\section{Introduction}
Let $\bfG$ be a reductive (connected) algebraic group defined over a non-Archimedean local field $F$.  Let $G=\bfG(F)$ \EitanD{be the corresponding $\ell$-group.}

%??e?The purpose of the present paper is to establish a connection between Cohen-Macaulay properties of objects in the category $\cM(G)$ of smooth representations of $G$ and more classical properties studied in Harmonic analysis on $p$-adic Homogenuous $G$-spaces.

In this paper we prove few results about the representation theory of $G$ and about its relation with unipotent subgroups of $G$. \EitanC{Our main result concerns density of spherical characters with respect to unipotent subgroups. 
%In order to formulate it recall 
}
%A key result is the following unpublished Theorem due to the Joseph Bernstein.
In order to formulate it we recall the \EitanD{construction} 
%concept 
of spherical character \EitanD{attached to an irreducible smooth representation.
We denote by $\cH(G)$ the Hecke algebra of $G$ consisting of smooth compactly supported measures on $G.$ We also let $C^{-\infty}(G)$ the space of generalized functions on $G.$ These are linear functionals on $\cH(G).$
}
\begin{defn}
	Let $(\pi,V)$ be a smooth \EitanE{(complex)} representation of $G.$
	%\in \cM(G)$. 
	Let $\pi^*$ be the \RamiJ{dual representation} of $\pi$ and \RamiJ{let} $\tilde \pi:=(\pi^{*})^{\infty}$  be the \RamiJ{contragredient  (i.e. smooth dual) representation} of $\pi$.
	For $l_1\in \pi^*$ and $l_2\in \tilde \pi^*$,  we define the  
	spherical character as the generalized function $\chi^{\pi}_{l_1,l_2}\in C^{-\infty}(G)$ by $$\langle \chi^{\pi}_{l_1,l_2} ,f \rangle:=\langle  f\cdot l_1,l_2\rangle,$$  for any smooth compactly supported measure $f \in \cH(G)$.

\end{defn}
Note that if \EitanE{$H_{1},H_{2} \subset G$ are two (closed) subgroups, $\chi_{i}:H_{i} \to \C$ are (continuous) characters and  $l_1\in(\pi^*)^{H_{1},\chi_{1}},l_2\in(\tilde \pi^*)^{H_{2},\chi_{2}}$\RamiN{,} then the spherical character $\chi^{\pi}_{l_1,l_2}$ lies  in the space of $(H_{1} \times H_{2}, \chi_{1} \times \chi_{2})$-equivariant generalized functions. We shall denote this space by $C^{-\infty}(G)^{H_{1}\times H_{2},\psi\times \psi}$.}

%From these theorems we deduce a result on
% The main result of this note
%To formulate our density result, we introduce the \RamiH{ad-hoc notion} of {\bf rich} collections of irreducible representations.
\EitanD{We also} recall the infinitesimal character map.
\EitanD{Let $\cM(G)$ be the category of smooth \EitanF{(complex)} representations of $G.$ We will denote by $\Irr(G)$ the set of isomorphism classes of irreducible smooth representations of $G.$}
Let $\fz(G)$ be the Bernstein  center of $\cM(G).$ Given an irreducible representation $(\pi,V) \in \Irr(G)$ the action of each $z \in \fz(G)$ on $V$ is, by Schur's lemma, given by a multiplication with a complex scalar $\chi_{(V,\pi)}(z) \in \C.$ Notice that 
$\chi_{(V,\pi)}: \fz(G) \to \C$ is an algebra homomorphism.  
\RamiJ{Denote by 
\begin{equation}\label{eq:theta}
\Theta(G):=Mor_{\C}(\fz(G),\C)
\end{equation}
the set of algebra homomorphisms of the 
%Bernstein's 
center.}

\RamiJ{W}e obtain the {\it infinitesimal character} map:
 
\begin{equation}\label{inf.char}
\inf: \Irr(G) \to \Theta(G)
\end{equation}
\EitanD{defined by}
%the association  
$\inf(V):=\chi_{(V,\pi)}.$

In words, $\inf(V)$ is the character by which $\fz(G),$ the Bernstein's 
center of $G$, acts on the irreducible representation $V$.
\begin{remark}
Let $A$ be a unital algebra over $\C.$ It is natural to consider the set of maximal ideal in $A$ or the set of prime ideals in $A$ as its spectrum. Our choice to consider $\Spec_{\C}(A):=Mor_{\C}(A,\C)$ is guided by the 
%example of \cite{LLS}\RamiQ{I did not reed this reference} and the 
fact that when $B=\Pi_{\alpha \in I} A_{\alpha}$ we obtain $\Spec_{\C}(B)=\bigsqcup_{\alpha \in I} \Spec_{\C}(A_{\alpha})$. {Such equality does not hold for the prime (or maximal) spectrum}. \EitanD{For more details one can consult \cite{LLS}}.
\end{remark}

\EitanD{To formulate our density result, we introduce the \RamiH{ad-hoc notion} of {\bf rich} collections of irreducible representations.}

\begin{defn}
% We denote by $\fz(G)$ the Bernstein  center of $\cM(G)$ and $\Theta(G)$ by the set of its characters.
We say that a set $\Pi\subset \Irr(G)$ (of smooth irreducible representations of $G$) is {\bf rich} if 
$\inf(\Pi) \subset \Theta(G)$ is Zariski dense.
%the corresponding set of characters of $\fz(G)$ attached %to $\Pi$ by the map $$Irr(G) \to \Theta(G)$$ is %Zariski dense.
%in $\Theta(G)$. 	
\end{defn}

%Recall that ${\bf G}$ is an algebraic group defined over a non-Archimedean field $F.$ Let ${\bf U}$ be a maximal unipotent subgroup of ${\bf G}$ and let $U={\bf U}(F).$

\EitanD{Our} density theorem \EitanD{for} spherical characters is the following:
\NextVer{2 different $\psi$}
\begin{introtheorem}[Density; see \S \ref{sec:dens} below]\label{thm:dens}
Let $U<G$ be a closed subgroup which is exhausted by its open compact subgroups (i.e. any compact subset of $U$ is contained in a compact subgroup of $U$).
	Let $\psi$ be  a character of $U$.
    Assume that $(G,U)$ is of finite type (see Definition \ref{def:ft} below). 
    	Let $\Pi\subset \Irr(G)$ be {\bf rich}. 

	%Let $\bfG,G,U,\psi$ and $F$ be as above. 
	
	Then the space of spherical characters $$\spa(\{\chi^{\pi}_{l_1,l_2}|\pi\in \Pi, l_1\in(\pi^*)^{U,\psi},l_2\in(\tilde \pi^*)^{U,\psi}\})$$ is dense in the space
	 $C^{-\infty}(G)^{U\times U,\psi\times \psi}$
\end{introtheorem}	

{
The following \RamiI{theorem} provides examples of pairs satisfying the conditions of the theorem above. 
\begin{theorem}\label{thm:BH}
\EitanD{Let ${\bf U}$ be a maximal unipotent subgroup of ${\bf G}$ and let $U={\bf U}(F).$}
%    Let $U<G$ be the unipotent radical of a minimal parabolic $P_0<G.$ 
Then $(G,U)$ is of finite type. Namely, for any character $\psi:U \to \C^{\times}$ and any irreducible $\pi \in \Irr(G)$ we have:
    $$\dim (\pi^{*})^{U,\psi} < \infty$$
\end{theorem}
\RamiO{
This theorem can be easily deduced from the results of \cite{BH03}, see Appendix \ref{app:BH} for more details. 
}
}

%\begin{introtheorem}[Density]\label{thm:dens}
%	Let $\bfG,G,U,\psi$ and $F$ be as above. 
%	Assume that  $U=\bfU(F)$ with $\bfU$ %maximal unipotent subgroup of $\bfG$.%
%	Let $\fz(G)$ be the Bernstein  center of $\cM(G)$ and $\Theta(G)$ be the set of its characters.
%	Let $\Pi\subset \irr(G)$ be a set of smooth irreducible representations of $G$ such that the corresponding set of characters of $\fz(G)$ is Zariski dense in $\Theta(G)$. 	
%	 Then $$\spa(\{\chi^{\rho}_{l_1,l_2}|\rho\in \Pi, l_1\in(\rho^*)^{U,\psi},l_2\in(\tilde \rho^*)^{U,\psi}\})$$ is dense in 
	 %$C^{-\infty}(G)^{U\times U,\psi\times \psi}$
%\end{introtheorem}	

The proof of Theorem \ref{thm:dens} is based on the following two results.

\begin{introtheorem}[See \S \ref{sec:CM} below]\label{thm:CM}
The category $\cM(G)$ of smooth representations of $G$ is \CM, %\Rami{of full dimension}
i.e. for any finitely generated projective module $P\in \cM(G)$ the algebra 
$End(P)$ is a \CM\ module \Rami{of full dimension} \RamiJ{(see Definition \ref{def:cm} below)} over its center.  
\end{introtheorem}

%The proof is due to Bernstein. 
%It was written in the \cite{AS_CM} and for completeness we include it here.

%The proof is presented i
%For $G$ quasi-split, with $U \subset G$ the  maximal unipotent subgroups in $G.$ The small Whittaker was studied in \cite{BH03} where it is shown that it is locally finitely generated over the Bernstein center (see Theorem 4.2 of \cite{BH03}). \RamiQ{wasn't there another result about projectivity}

%\RamiH{Next, we prove the the following:}
\begin{introtheorem}[See \S \ref{sec:proj} below]\label{thm:proj}
	%Let $\bfU<\bfG$ be  unipotent subgroup. 
	%Let $\Pi\subset \Irr(G)$ be {\bf rich}. 
		%Let $U<G$ be a closed subgroup which is exhausted by its open compact subgroups (i.e. any compact subset of $U$ is included in an open compact subgroup of $U$).
%	Let $\psi$ be  a character of $U$.
    %Assume that $(G,U)$ is of finite type (see Definition \ref{def:ft} below), then for any character  $\psi$ of $U$ 
Let $\bfG,G,U,\psi$ and $F$ be as in Theorem \ref{thm:dens}. 

    Then
    the module $ind_{U}^G(\psi)$ is a projective object in the category $\cM(G) $ of smooth representations of $G$.	 
\end{introtheorem}

Using the Bernstein decomposition of the category $\mathcal{M}(G)$, the proof of Theorem \ref{thm:proj} is easy.  It is merely the observation that 
%The exactness of the functor of co-invariance is known and all we need to observe that 
\EitanF{finitely} presented flat modules \RamiJ{are} projective.
%, if  implies projectivity result.

%The density of spherical characters relative to a unipotent subgroup is deduced from these two results.

\subsection{Related Works}
\subsubsection{\RamiO{Density of spherical characters}}
\RamiJ{Density of characters of irreducible representations inside the space of $Ad$-invariant distribution, is  a classical result in representation theory (See \cite{DKV84,Kaz86}). In \cite{AGS_Z} a general result about density of spherical characters of \emph{admissible} representations was proven. Theorem \ref{thm:dens} give a very strong kind of density for the unipotent case. We do not expect this kind of density to hold in general.}

\subsubsection{\RamiO{Cohen-Macaulay property of $\cM(G)$}}
\EitanC{
\RamiJ{A proof of} Theorem \ref{thm:CM}  was \RamiJ{recently} published as Proposition 3.1 of \cite{BBK}. For completeness we include the proof here. }
%that uses the summary (given in the second section of the present paper) of the Bernstein decomposition theory. Otherwise, the argument is identical. 
{

\subsubsection{\RamiO{Projectivity of $ind_{U}^{G}(\psi)$}}
\RamiJ{For non-degenerate $\psi$, t}he small Whittaker module $\mathcal W=ind_{U}^{G}(\psi)$
is studied in \cite{BH03}. In particular it is shown to be finitely generated over any component of $\Theta(G)$, (\cite[Theorem 4.2,(2)]{BH03}).
%(in the language we employ here)}
%This is their Theorem 4.2, item (2). 
%Here we use this finite generation but deduce it from \cite{AAG}.
}

In recent works, the small Whittaker module $\mathcal W$ was studied further and\RamiH{, in the quasi-split case, its projectivity was shown in \cite[Corollary A.6]{CS19}}. \RamiH{Note that the proof in \cite{CS19} can be easily generalized to cover Theorem \ref{thm:proj}}.

Our argument here is \RamiH{slightly more general,} it is based on Bernstein's decomposition of $\cM(G)$ (showing in particular that $\cM(G)$ is locally Noetherian) and a general result \RamiH{from} homological algebra connecting flatness to projectivity for finitely presented modules.
\NextVer{structure of the paper}
\subsection{Acknowledgments}

\Rami{We} would like to thank Guy Henniart for encouraging us to write \RamiN{Appendix \ref{app:BH}}, \EitanE{for his help in \RamiN{Appendix \ref{app:deg}}, for explaining us the proof of Lemma \ref{lem:der} and for many useful suggestions regarding the exposition of the paper}. We also thank Elad Sayag for a useful conversation concerning \RamiN{Theorem \ref{thm:proj}}.

\section{The Bernstein decomposition of the category $\cM(G)$ and the Bernstein center}
In this section we summarize the parts of the theory of the Bernstein center and Bernstein decomposition that are used in this paper.

\EitanD{
Recall that $\bfG$ is a reductive (connected) algebraic group defined over a non-Archimedean local field $F$ and  $G=\bfG(F)$.
}

\Rami{Throughout \EitanD{the body of} the paper we fix a minimal parabolic subgroup $\RamiM{\bf B}\subset \bf G$ and its Levi subgroup $\RamiM{\bf T}$.  
All the parabolic subgroups and Levi subgroups that we consider are assumed to contain $\RamiM{\bf T}$ (even if we do not say that explicitly).  
\EitanD{Note that any parabolic subgroup that contains $\RamiM{\bf T}$ has a unique Levi subgroup $M$ that contains $\RamiM{\bf T}.$}
%satisfying this condition.

By parabolic and Levi subgroups of $G$  we mean 
the groups of $F$-points of parabolic and Levi subgroups of $\bf G$.
{
We now list} few notations to formulate some results from the theory of Bernstein center.
}
%the result.
	\begin{notation}
		Let $\bf P$ be a \RamiG{parabolic subgroup} of $\bf G$ and  $\bfM$ be its Levi \Rami{subgroup}. Let $\rho$ be a cuspidal representation of $M=\bfM(F)$. Denote
		\begin{itemize}
			\item $\RamiE{i^G_M}: \cM(M) \to \cM(G)$ the (normalized) parabolic induction {from} $M$ with respect to $P.$
			\item $ \RamiE{\overline i^G_M}(\rho): \cM(M) \to \cM(G)$   the (normalized) parabolic induction from $M$ with respect \RamiJ{to the} opposite \RamiG{parabolic subgroup}  $\bar \bfP$.
			\item
			$\RamiE{r^G_M}:  \cM(G) \to \cM(M)$  the (normalized) Jacquet functor.
			\item
			$\RamiE{\overline r^G_M}:  \cM(G) \to \cM(M)$  the (normalized) Jacquet functor
			with respect \RamiJ{to the} opposite \RamiG{parabolic subgroup}  $\bar \bfP$.
			\EitanE{\item We denote by $G^{0}$ the subgroup of $G$ generared by compact subgroups. }
			\item \EitanD{For \Rami{a Levi subgroup}  $M \subset G$ we denote by $\fX_M$ the \EitanE{complex} torus of unramified characters on $M.$ \EitanE{Note that the ring $O(\fX_{M})$ of regular functions on $\fX_{M}$ is isomorphic to the group algebra $\C[M/M^{0}]$.}
			\item \EitanE{By a cuspidal data we mean a pair $(M,\rho)$ consisting of a Levi subgroup $M \subset G$ and its irreducible cuspidal representation $\rho$.} } 
		\end{itemize}
	\end{notation}
     
     %It will be denoted by $\fX_M.$ 
     
%\begin{remark}
	%Given \Rami{a Levi subgroup} $M\subset G$, the objects defined above depend on the choice of a parabolic $P \subset G.$ However, unless stated otherwise the result will not depend on this choice and we will ignore it.
%\end{remark}
	
	%\begin{definitin}
	%\begin{itemize}
	    %\item A cuspidal data for $\bfG$ is a pair $(\bfM,\rho)$ where $\bfM$ is %a levi subgroup of $\bfG$ and $\rho \in \irr(\bfM(F))$ is a cuspidal %representation.
	    %\item 2 cuspidal dataare equvuvalent iff they are 
	%\end{itemize}
	%\end{definitin}

	The following theorems \EitanD{summarizes} 
	%is a summary of 
	the results proved in {\cite{Ber_cent,Ber_sec,Ber_padic,Bus}}. 
	See  {\cite[Theorem 2.5]{AS_hom}} and {\cite[\S\S 2.1]{AAG}} for detailed references for each item.
\subsection{\Rami{Bernstein Decomposition}}	

\begin{theorem}\label{thm:cent:dec}
	$ $
		\begin{enumerate}

			\item\label{Ber.cent:2.5} {The f}unctor $\RamiE{r^G_M}$ is right adjoint to $\RamiE{\overline i^G_M}$, that is $$ {\Hom_G}(\RamiE{\overline i^G_M}(V),W) \cong  {\Hom_M}(V,r^G_M(W)).$$
			
			\item\label{Ber.cent:3}  For a {cuspidal \EitanD{data}} $(M,\rho)$ let \EitanB{$\Psi_{G}(M,\rho)=\RamiE{\overline i^G_M}(\rho \otimes O(\fX_M))$} be the normalized parabolic induction of \EitanE{$\rho \otimes O(\fX_M)$, where} the action of $M$ is diagonal. Then $\Psi_{G}(M,\rho)\in \cM(G)$ is a projective generator of a direct summand of the category $\cM(G).$

			\item\label{Ber.cent:4}  We call the category generated by $\Psi_{G}(M,\rho)$ a Bernstein block of $\cM(G)$. The collection of those blocks is denoted by $\Omega$. For $\omega\in \Omega$  and $V\in \cM(G)$ we denote by $V_\omega\in \omega$ to be the corresponding direct summand. We have $$\cM(G)=\EitanD{\prod_{\omega\in \Omega}} \omega.$$
			
			\item \label{descriptionofcenter} \EitanF{F}or a cuspidal data $(\rho,M)$, the Bernstein center $\fz(G)$ acts through a character $\chi_{(M,\rho)}$ on $i_M^G(\rho)$. This gives a bijection between \RamiJ{the set} of conjugacy classes of cuspidal data and the \RamiJ{variety} $\Theta(G)$ \RamiJ{(see \eqref{eq:theta})}.
		\end{enumerate}
		
	\end{theorem}
	\Rami{The following gives a description of the algebra \RamiI{$\End_G(\Psi_{G}(M,\rho)))$}  and its center.}
\begin{theorem}\label{thm:cent:blk}
	$ $
		\begin{enumerate}
			\item \label{thm:ber:azu0}  
			Let $\rho$ be \RamiG{an irreducible cuspidal} representation of $G$, and set \Eitan{$ \fI_\rho:=\{\psi\in\fX_G|\psi \rho\simeq \rho\}$}  and embed $ {O}(\fX_G)$ in $\cR_{(G,\rho)}:=\End(\Psi_{G}(G,\rho))$.
			Then there exist\RamiE{s} a decomposition $$\cR_{(G,\rho)}= \bigoplus_{\psi\in \fI_\rho}O(\fX_G)\nu_\psi,$$ \RamiG{such that}  $\nu_\psi f=f_\psi \nu_\psi $ and $\nu_\psi\nu_{\psi'}=c_{\psi,\psi'}\nu_{\psi\psi'}$ where $f_\psi$ is  the translation of $f\in O(\fX_G)$ by $\psi$ and  $c_{\psi,\psi'}$ are scalars. 
			
			In particular $Z(\cR_{(G,\rho)}) \cong O(\fX_G)^{\fI_\rho}\cong O(\fX_G/\fI_\rho)$
			\item \label{thm:ber:ind} 
			\Rami{For a cuspidal data $(\rho,M)$,
			t}he natural embeding $$ End_M(\Psi_{M}(M,\rho))) \to End_G(\Psi_{G}(M,\rho))$$ gives an isomorphism
			$$ Z(End_M(\Psi_{M}(M,\rho)))^{W_{M,\rho,G}} \cong Z(End_G(\Psi_{G}(M,\rho)))$$ where $W_{M,\rho, G}=\{g\in G|g M g^{-1}=M; \, \rho|_{\EitanF{M^0}}\circ Ad(g)\cong  \rho|_{\RamiM{M^0}}\}$. cf. \cite[pp 73-74]{Ber_padic}.

		\end{enumerate}
		
	\end{theorem}

\subsection{\Rami{Splitting subgroups}}

	\begin{theorem}\label{thm:cent:spt}
	%[{see {\cite{Ber_cent}};  or  {\cite[Theorem 2.5]{AS_hom}} and {\cite[\S\S 2.1]{AAG}}}]
		$ $
		\begin{enumerate}
			\item\label{Ber.cent:basis}
			There exists a local base $\mathcal{B}$ of the topology at the unit $e \in G,$ such that every $K \in \cB$ is a compact (open) subgroup satisfying the following:
			
			\begin{enumerate}
				
				\item \label{Ber.cent:1}
				
				The category $\cM(G,K)$ of representations generated by their $K$-fixed vectors is a direct summand of \EitanD{the category} $\cM(G).$
				
				\item \label{Ber.cent:2}
				
				The functor $V \to V^{K}$ is an equivalence of categories from $\cM(G,K)$ to the category of modules over the Hecke algebra  $\cH(G,K)$.
				\item\label{Ber.cent:not} The algebra  $\cH(G,K)$ is Noetherian.
				
				\item For any \RamiJ{(standard)} Levi subgroup $M \subset G$ and for any $V \in \cM(G)$ the map
				$$V^K\to r_{M}^G(V)^{K\cap M}$$ is onto.
			\end{enumerate}
			We will call an open compact subgroup $K$ {satisfying} these properties, a splitting subgroup.

			\item\label{Ber.cent:spt4}  For any splitting subgroup  $K\subset G$ we have $$\cM(G,K)=\bigoplus_{\omega\in \Omega_K} \omega$$ for some \EitanD{finite} subset $\Omega_K \subset \Omega$\EitanF{.}
            \item \label{Ber.cent:5}We have $$\bigcup_{K\in \cB} \Omega_K=\Omega$$
			
		\end{enumerate}
		
	\end{theorem}

\subsection{\Rami{Relation to harmonic analysis}}	
%\begin{definition}
%For a cuspidal data $(M,\rho)$, we call the category generated by $\Psi_{G}(M,\rho)$ a Bernstein block of $\cM(G)$. The collection of those blocks is denoted by $\Omega_G$. For $\omega\in \Omega$  and $V\in \cM(G)$ we denote by $V_\omega\in \omega$ to be the corresponding direct summand.
%\end{definition}
We recall the following definition of pairs of finite types and some of their properties.
\begin{definition}[{cf. \cite[\S 3]{AS_hom}}]\label{def:ft}

Let $H\subset G$ be a closed subgroup.
We say that $(G,H)$ is of finite type if for any character $\chi:H \to \C^{\times}$ and any $\pi\in \Irr(G)$   we have $$\dim((\pi^*)^{H,\chi})<\infty$$
\end{definition}
\begin{thm}[{\cite[Theorem B.0.2]{AGS_Z}}]\label{thm:ft}
If $(G,H)$ is of finite type  then  for any  character $\chi:H \to \C^{\times}$ 
%the following:
%\begin{itemize}
%\item 
and any compact open $K < G$ the module  $ind_H^G(\chi)^K$ is finitely generated $\cH(G,K)$-module.
%\item for any $\omega \in \Omega(G)$ the representation $ind_H^G(\chi)_\omega$ is finitely generated.
%\end{itemize}
\end{thm}
\begin{cor}\label{cor:fg}
If $G,H,\chi$ are as above, then for any $\omega \in \Omega(G)$ the representation $ind_H^G(\chi)_\omega$ 
%\EitanD{see Theorem 2.2. item 3} is 
is finitely generated.
\end{cor}
\begin{proof}
This follows immediately from the previous theorem (Theorem \ref{thm:ft}) and from \Rami{Theorem \ref{thm:cent:spt}   
(\ref{Ber.cent:spt4}, \ref{Ber.cent:5}).}
\end{proof}
\Rami{
Theorem \ref{thm:BH} provides examples of pairs of finite type. 
%It was essentially proved in \cite{BH03}. See Appendix \ref{app:BH} for more details. 
%\begin{theorem}
%    Let $U<G$ be the unipotent radical of a minimal parabolic $P_0<G.$ Then $(G,U)$ is of finite type. 
%\end{theorem}
}

\section{Cohen-Macaulay property of the category of smooth representations of reductive p-adic groups}\label{sec:CM}
In this section we prove {Theorem \ref{thm:CM}. \Rami{
We first fix conventions regarding the concept of Cohen-Macaulay modules. \RamiJ{
We use the notion of Cohen-Macaulay modules over local rings, and their dimension. Since the dimension of a module  over non-local ring might vary from point to point, there are  several useful version of the notion of Cohen-Macaulay module in this more general situation. In the following we consider some of these versions.}

\begin{definition}\label{def:cm}
Let $A$ be a commutative unital \EitanE{$\C$-algebra} and let $V$ be an $A$-module.
\begin{itemize}
    \item We say that $V$ is locally Cohen-Macaulay if for any $p \in \Spec(A)$, the module $V_{p}$ is Cohen-Macaulay over $A_{p}.$
    \item We say that $V$ is 
    \RamiJ{equi-dimensional}  Cohen-Macaulay 
    \RamiJ{module}
    if there is an integer $d \geq 0$ such that for any $p \in \Spec(A)$, the module $V_{p}$ is Cohen-Macaulay module of dimension $d$ over $A_{p}.$
    \item We say that $V$ is Cohen-Macaulay of full dimension if %there is an integer $d \geq 0$ such
    %that 
    for any $p \in \Spec(A)$, the module $V_{p}$ is Cohen-Macaulay module of dimension $\dim_{p}(\Spec(A)).$
    \item We say that $V$ is component\RamiJ{-}wise Cohen-Macaulay if for any \RamiI{connected} component $X \subset \Spec(A)$ the restriction $M|_{X}$ is \RamiJ{equi-dimensional Cohen-Macaulay module} over $O_X(X).$
\end{itemize}
\end{definition}

We note that the following implications holds:
\begin{itemize}
    \item \RamiJ{equi-dimensional} Cohen-Macaulay module is component\RamiJ{-}wise Cohen-Macaulay module. 
    \item Component wise Cohen-Macaulay module is locally Cohen-Macaulay module.
    \item Full dimension Cohen-Macaulay module is component\RamiJ{-}wise Cohen-Macaulay.  
\end{itemize}
}
Let us recall the formulation of Theorem \ref{thm:CM}}:
%we provide a proof of the following unpublished result of Joseph Bernstein.
%'s theorem show the ??CM2.4.2.
%\RamiQ{define localy CM, globaly CM and CM of full dim}
\begin{theorem}[{See also \cite[Proposition 3.1]{BBK}}]\label{thm:CM2}
The category $\cM(G)$ of smooth representations of $G$ is \CM, that is, for any finitely generated projective module $P\in \cM(G)$ the algebra \Rami{
$\End(P)$ is a \CM\ module of full dimension over its center.}  
\end{theorem}

%\begin{remark} We dont need you anymore. Sorry.
%\Rami{The proof below will show that $\End(P)$ is Cohen-Macaulay of full dimension. In particular, component wise Cohen-Macaulay.}
%\end{remark}

%\begin{theorem}
%%Let $\bfG,G$ and $F$ be as above. 
%%Then t
%The category $\cM(G)$ of smooth (complex) $G$-modules is (locally) Cohen-Macaully. In other words, for any projective f.g. $P\in \cM(G)$ the algebra  $End(P)$ is a finitely generated (locally) Cohen-Macaully module over its center\RamiQ{I do not think that we need the rest of the sentens} which is finitely generated over $\C$.
%\end{theorem}	
%\RamiQ{to add a remark on block with CM}

\begin{proof}
Using Theorem 	\EitanE{\ref{thm:cent:dec}(\ref{Ber.cent:3},\ref{Ber.cent:4})}, without loss of generality, we may assume that \EitanE{$P=\Psi_{G}(M,\rho) \cong i_M^G(ind_{M^0}^M(\rho|_{M^0}))$}, where $M<G$ is a Levi subgroup and $\rho\in \Irr(M)$ is a cuspidal representation. We have 
	\begin{align*}
	\End(P)&=\Hom_{G}(i_M^G(ind_{M^0}^M(\rho|_{M^0})),i_M^G(ind_{M^0}^M(\rho|_{M^0})))=\\&=
	\Hom_M(ind_{M^0}^M(\rho|_{M^0}),\bar r_M^G(i_M^G(ind_{M^0}^M(\rho|_{M^0}))))
	\end{align*}
	
	Let $R=ind_{M^0}^M(\rho|_{M^0})$ \EitanF{and} $Q:=\bar r_M^G(i_M^G(ind_{M^0}^M(\rho|_{M^0})))$.
	\EitanE{Note that $R$ is finitely generated and hence, by \cite[Proposition 33]{Ber_padic}, $Q$ is also finitely generated module.} By \Rami{Theorem \ref{thm:cent:blk}\eqref{thm:ber:ind}} and \cite[Theorem 2.1]{BBG97}, 
	%(, CM base change), 
	it is enough to show that $\Hom_M(R,Q)$ is finitely generated Cohen-Macaulay module over the center of $\End_M(R)$. For this it is enough to show that 
	 $\Hom_M(R,Q)$ is finitely generated Cohen-Macaulay over $\C[M/M^0]$.
	By \Rami{Theorem \ref{thm:cent:dec}\RamiI{(\ref{Ber.cent:2.5}, \ref{Ber.cent:3})}} $Q$ is projective. Hence, by \Rami{Theorem \ref{thm:cent:dec}\eqref{Ber.cent:3}} we can decompose $Q=Q'\oplus Q''$ where $Q'$ is a direct summand  of $R^n$ (for some $n$) and $Hom_{\EitanF{M}} (R,Q'')=0$. Thus, by \cite[Theorem 2.1]{BBG97},
	%CM base change
	it is enough to show that $\End_{M}(R)$ is finitely generated Cohen-Macaulay \Rami{module of full dimension} over $\C[M/M^0]$. This follows from  \Rami{Theorem \ref{thm:cent:blk}\eqref{thm:ber:azu0}}.
\end{proof}

\section{Projectivity of $ind_{U}^G(\psi)$ -- proof of Theorem
\ref{thm:proj}
}\label{sec:proj}

%\RamiQ{erazed the historical part}

%The work of \cite{BH03} examined the \RamiH{small Whittaker model}  from the perspective of Bernstein decomposition. Among other things, they showed the following:

%\begin{lemma}\label{lemma: fg}
%Let $G$ be a quasi split group and let $U<G$ be the unipotent radical of a Borel subgroup $B<G.$ Let $\psi$ be a non-degenerate character of $U$. Let $\omega \in \Omega$ be a Bernstein component of $G.$ Then $ind_{U}^G(\psi)_{\omega}$ is finitely generated. 
%\end{lemma}
%Recently, it was shown that the  \RamiH{small Whittaker model} representation of a quasi split reductive group is projective (\cite{CS19}, appendix). In this section we provide a generalization of this result that is valid for degenerate Whittaker models. 
%Our argument is avoid some of the specific properties of the Gelfand Graev model. 
%This was done recently in \cite{MR3910471} using a different argument. 

In this section we prove Theorem \ref{thm:proj}. For this we first  \RamiN{prove flatness of the module}  $ind_{U}^G(\psi)$.
\subsection{Flatness of $ind_{U}^G(\psi)$}\label{ssec:elad}
%\begin{defn}
%We say that $K<G$ is a splitting subgroup if the set %$Irr_K(G)=\{V \in Irr(G): V^{K} \e \{0\}\}$ splits $\cM_{K}(G).$ 
%Namely, for any $V \in cM(G)$ we have $V=W(K) \oplus %W(K)^{\purp}$ where $JH(W(K)) \subset Irr_K(G)$ and %
%$JH(W(K)^{\purp}) \cup %Irr_K(G) = \empty.$
%\end{defn}
%??	
%a reformulation of a Lemma of Jacquet.}

{
\begin{proposition}\label{thm0:flat}
    \RamiN{Let $L$ be an $\ell$-group and l}et $U<\RamiN{L}$ be a closed subgroup that is exhausted by its open compact subgroups.
	Let $\psi$ be  a character of $U$. Then $\cW:=ind_{U}^\RamiN{L}(\psi)$  is a flat $\cH(\RamiN{L})$-module.
\end{proposition}
\begin{proof}
\EitanF{C}onsider $\cW$  as a right representation via the inversion on $\RamiN{L}$.
	Consider also the functor  $$F:\cM({\cH(\RamiN{L})})\to Vect$$ given by $$F(M)=ind_{U}^{\RamiN{L}}(\psi) \otimes_{{\cH(\RamiN{L})}} M.$$ 
	We have to show that $F$ is exact.
	
		Let $$W:\cM(\RamiN{L})\to Vect$$ be the functor given by $$W(\pi)=\pi_{U,\psi^{-1}}.$$ Since $U$ is exhausted by its open compact subgroups, {the Jacquet's lemma (see e.g. \cite[Lemma 2.35]{BZ}) implies that \EitanF{$W$ is an exact functor.}} 
	
	Let $M \in \mathcal{M}(\mathcal{H}(\RamiN{L}))$, using
	$$ind_{U}^{\RamiN{L}}(\psi)\cong \psi \otimes_{\cH(U)}\cH(\RamiN{L}),$$
	we obtain
	$$F(M) \cong (\psi \otimes_{\cH(U)}\cH(\RamiN{L})) \otimes_{\cH(\RamiN{L})} M \cong M_{U,\psi}=W(M).$$
 	Hence
 	$$ F  \cong W.$$
	The theorem follows.
%	??Appendix B \cite{AS20}
\end{proof}
This theorem together with {Theorem \ref{thm:cent:spt}\eqref{Ber.cent:2}} gives us the following:

\begin{cor}\label{thm:flat}
	%\RamiQ{do not need to recall the setting} Let $\bfG$ be a reductive algebraic group defined over a non-Archimedian local field $F$.  Let $G=\bfG(F)$.
	%Let $\bfU<\bfG$ be  unipotent subgroup. 
	%Let $U<G$ be a closed subgroup that is exhausted by its open compact subgroups.
	%Let $\psi$ be  a character of $U$. 
		%Let $\cW:=ind_{U}^G(\psi)$ 
		%and consider it as a right representation via the inversion on $G$.
		\RamiN{Let $U<G$ be a closed subgroup that is exhausted by its open compact subgroups.	Let $\psi$ be  a character of $U$ and $\cW:=ind_{U}^\RamiN{L}(\psi)$.}
		%Let $U,\psi$ and $\cW$ be as in the previous theorem.
		Let $K<G$ be a splitting subgroup. Then the right  ${\cH(G,K)}$-module $\cW^K$ is flat.
	%\EitanD{as a $\cH(G,K)$ module}
\end{cor}
}

\subsection{Proof of projectivity \RamiJ{of $ind_{U}^{G}(\psi)$} (Theorem \ref{thm:proj})}
\begin{proof}[Proof of Theorem \ref{thm:proj}]
	By assumption\EitanF{,} \EitanD{the pair $(G,U)$ is of finite type.}
%	the module 
\EitanD{Let $\cW=ind_{U}^{G}(\psi).$ By}
%is of finite multiplicity. 
\RamiH{Theorem} \ref{thm:ft}
%	\cite[Appendix B]{AGS_Z}
\RamiI{the module}
$\cW^K$ is finitely generated over $\Rami{\cH(G,K)}$.
%	{\bf NO Problem: Def 2.4 plus Thm 2.5 helps here}
	The algebra $\Rami{\cH(G,K)}$ is Noetherian (by \Rami{Theorem \ref{thm:cent:spt}\eqref{Ber.cent:not})} so the module $\cW^K$ is finitely presented. Combining with Theorem \ref{thm:flat}\EitanF{,} we see that $\cW^K$ is a finitely presented flat module. It is well known that a finitely presented flat module is projective (e.g.  
	\cite[Theorem 3.56]{Rot09}).
	%By and by a standard result in algebra 
	%(see Theorem 3.56 in 
	%\cite{Rot09})
	%MR2455920
	It follows that $\cW^K$ is a projective $\Rami{\cH(G,K)}$-module. Thus it is a projective object in $\cM(\EitanF{\cH(G,K)})$. 
	%generated module
	%Thus, by Theorem 3.56 in 
	%\cite{Rot09}
	
%	{Rot}??
%	https://math.stackexchange.com/questions/2228441/finitely-presented-module-flat-implies-projective
%	??
%	Theorem 3.56 in Rotman's Homological Algebra
	
	%Theorem \ref{thm:flat} implies that $\cW^K$ is projective $\Rami{\cH(G,K)}$-module. Thus it is a projective object in $\cM(\cH_(G))$.
	For a splitting subgroup $K \subset G$
	we let $$I_{K}: \mathcal{M}(\mathcal{H}(G,K))
	\to 
	\EitanF{\mathcal{M}(G,K)} 
	$$
	be the equivalence of categories as in Theorem \ref{thm:cent:spt}. 
	%we get that $I(\cW^K)$ is a projective object in $\cM(G)$.
	For a pair of splitting subgroups $K<L<G$ and $V\in \cM(G)$ we can decompose $V^K=I_L(V^L)^K\oplus V_{K,L}$.
	By \Rami{Theorem \ref{thm:cent:spt}\eqref{Ber.cent:1}} we can choose a descending sequence $K_n$ of  splitting subgroups that forms a basis for the topology at $1$. We get $$\cW=\bigoplus_n I_{K_{\EitanF{n}}}( \cW_{K_n,K_{n-1}}).$$ 
	Note that $\cW_{K_n,K_{n-1}}$ is a direct summand of $\cW^{\EitanF{K_{n}}}$  and hence projective. We obtain that  $I_{K_n}( \cW_{K_n,K_{n-1}})$ is a projective object in $\mathcal{M}(G).$ Finally, being a direct sum of projective objects,  $\cW$ is also a projective object of $\mathcal{M}(G).$
	
\end{proof}
\section{Density of spherical characters -- proof of Theorem \ref{thm:dens}}\label{sec:dens}
\subsection{\RamiK{Sketch of the proof}}
\RamiK{
We consider $V:=\Sc(G)_{U\times U,\psi\times \psi}$ as a module over $\fz(G)$ and thus as a sheaf over $\Theta(G)$. We have to show the vanishing of any section $v\in V$ 
that satisfy the identities $$\langle \xi, v\rangle=0$$  for any spherical character $\xi$ of any $\pi \in \Pi$.
 
 We do it by three steps:
 \begin{enumerate}
     \item \label{step:1} We show that for any such section $v$, there exist a Zariski dense subset $C\subset \Theta(G)$  such that for any $\chi\in C$ we have $v|_\chi=0$.
     \item \label{step:2} We show that $V$ is (locally) finitely generated Cohen-Macaulay module of full dimension (over each component).
     \item \label{step:3} We show that, given an element $v$ of a Cohen-Macaulay module of full dimension $M$ over a domain $A$, if  there exist a Zariski dense subset $C\subset \spec(A)$ such that for any $x\in C$, then we have $v|_x=0$ then $v=0$.
 \end{enumerate}
 The proof of Step \eqref{step:1} is in \S\S \ref{ssec:good}. For this we identify an open dense subset of  $\Theta(G)$  such that the category $\cM$ have \EitanF{a} very simple structure over points in this set (see Definition \ref{def:good} and Proposition \ref{cor:M.chi.vect}).
 
 Step \eqref{step:2} is based on Theorems \ref{thm:CM} and \ref{thm:proj}. It is proven in Lemma \ref{lem:wit.CM}.
 
The proof of Step \eqref{step:3} is in \S\S \ref{ssec:rel.tor}. }
\subsection{\RamiK{Good characters}}\label{ssec:good}
\RamiK{In this subsection we identify an open dense subset $\cU_G$ of the set $\Theta(G)$ of characters of  $\fz(G)$ such that the category $\cM$ have very simple structure over these characters.}

\begin{defn}\label{def:good}
We say that a cuspidal \EitanD{data $(\rho,M)$} is good if the following holds
    \begin{itemize}
        \item $i_M^G(\rho)$ is irreducible.
        \item  $(W_{M,\rho,G})_\rho=M$. {Here $(W_{M,\rho,G})_\rho$ is the stabilizer of $\rho$ under the action of the  group $W_{M,\rho,G}$ defined in \Rami{Theorem \ref{thm:cent:blk}(\ref{thm:ber:ind})}}.
        %the character $\chi \in \Theta(G)$  corsponding to $(M,\rho)$ is a regular value o
    \end{itemize}
    \RamiH{
If $(M,\rho)$ is good, we say that $\chi_{(M,\rho)}\in \Theta(G)$ is good, where $\chi_{(M,\rho)}$ is defined in \Rami{Theorem \ref{thm:cent:dec}} \eqref{descriptionofcenter}.

We denote by $\cU_G\subset \Theta(G)$ the set of good characters.}
\end{defn}

\begin{notation}
\Rami{Let $\cC$ be an \RamiK{$\C$-}abelian category and $\fz$ be its center. Given a character $\chi:\fz \to \C$ we define \RamiJ{a full subcategory:} $$\cC|_\chi:=\{c\in \EitanD{Ob(\cC)}| \forall \lambda \in \fz \text{ we have } \lambda|_{c}=\chi(\lambda)Id_{c}\}$$}
\end{notation}

\begin{prop}\label{cor:M.chi.vect}
$ $
\Rami{
\begin{enumerate}
    \item\label{cor:M.chi.vect:1} The set $\cU_G\subset \Theta(G)$ is \EitanF{a} Zariski open dense.
    \item\label{cor:M.chi.vect:2} $\forall \chi \in \cU_G$ we have $ \cM(G)|_{\chi}\cong Vect$. \RamiK{In other words: there exist a unique (up to an isomorphism) irreducible representation $\rho\in \cM(G)|_{\chi}$, and  any representation in $\cM(G)|_{\chi}$ is direct sum of several copies of $\rho$.}
\end{enumerate}
}%The set $\cU_G\subset \Theta(G)$ is
%Zariski open dense and 
%$\forall \chi \in \cU_G$ we have $ \cM(G)|_{\chi}\cong Vect$
\end{prop}
\RamiK{In order to prove this \EitanD{proposition} we need some preparations.}

\RamiI{Recall that any \RamiK{$\C$-}abelian category $\cC$ is a module category over the category of modules over its center $\fz$. In particular, for an object  $M \in Ob(\cC)$ and a character \RamiK{$\chi:\fz\to \C$} one can define $M \otimes_\fz \chi$. \RamiK{Explicitly, it can be described as a} quotient of $M$
by the sum
$$\sum_{\lambda \in \fz}\Ker(\lambda|_M-\chi(\lambda) id_M)$$
%by the sum over $\lambda$ of the kernels of %$\lambda|_M-\chi(\lambda) id_M:M\to M,$ when %$\lambda$ is ranging over $\fz$
}

\begin{lemma}\label{ProjGenSpecialization}
%	Let $\cC$ be an abelian category and $\fz$ be its center. 
	%Let $\fm\lhd \fz $ be a maximal ideal. 
	\Rami{Let $\cC, \fz$ and $\chi$ be as above.} Let $P$ be a projective generator of $\cC$. Then \Rami{$P\otimes_{\z} \chi$} %\RamiQ{Should be defined} 
	is a projective generator of  \Rami{$\cC|_\chi$}
\end{lemma}
\begin{proof}
\RamiK{Let $X\in \cC|_\chi$. We have $$\Hom_{\cC|_\chi}(\Rami{P\otimes_{\z} \chi},X)\cong \Hom_{\cC}(P,X).$$
This implies the assertion.} 

%\RamiQ{put less details}
%Let $P$ be a projective generator of $\cC.$
%\Rami{Let $\fm :=\Ker \chi$}

%Consider the functor $F:\cC|_\chi \to Vect$ defined 

%on $\cC|_\fm$ by
%That $P/\fm P$ is projective follows from

%by
%$$F(X)=Mor_{\cC|_\chi}(\Rami{P\otimes_{\z} \chi},X).$$
%We have $F(X) \cong Mor_{\cC}(P,X).$  Since $\cC|_\chi$ is a Serre subcategory \RamiQ{wrong} this implies that $F$ is exact \Rami{and hence $P\otimes_{\z} \chi$ is projective}. 
%Thus the projectivity of $P/\fm P$ is obvious. 
%This also proves that if $P$ is finitely generated then so is $P /\fm P,$ namely $Hom_{\cC|_\fm}(P /\fm P, \cdot)$ commutes with arbitrary direct sums.
%projective follows from
%To show that $\Rami{P\otimes_{\z} \chi}$ is a generator let $X \in \cC|_\chi.$ Consider $X$ as an object in $\cC.$
%$ as element 
%As $P$ is a generator of $\cC$ we have \Rami{an epimorphism:} $$\alpha: \bigoplus P \to X.$$ \Rami{This epimorphism factors through an epimorphism:} 
%The submodule $\fm (\bigoplus P)$ is mapped to $\fm X=0$ hence we obtain a map
%$${\bar \alpha}:\bigoplus (\Rami{P\otimes_{\z} \chi}) \to X.$$ \Rami{This finishes} the proof.
\end{proof}
\begin{lemma}[cf. Theorem 27 from \cite{Ber_padic}]\label{lem:gen.irr}
	Let \RamiH{$(\rho,M)$} be a cuspidal \EitanD{data} for $G.$ Then for a generic unramified character $\chi$ of $M$ the representation $i_M^G(\chi \rho)$ is irreducible.
\end{lemma}

\begin{lem}\label{lem:JH}
Let $(M,\rho)$ be  a cuspidal data. Let $w \in W_{M,\rho,G}$ then $i_M^G(\rho)$ and $i_M^G(Ad(w)\rho)$ have the same J\"{o}rdan-Holder components.
\end{lem}
%\RamiQ{to pospone to the end of the section}
For completeness, we include the proof of this Lemma in \S\S \ref{ssec:pf.JH}.

\begin{proof}[Proof of \EitanD{Proposition} \ref{cor:M.chi.vect}]
%    we say that a cuspidal datum $(M,\rho)$ is good if
%    \begin{itemize}
%        \item $i_M^G(\rho)$ is irreducible
%        \item if $(W_{M,\rho,G})_\rho=M$.
%        %the character $\chi \in \Theta(G)$  corsponding to $(M,\rho)$ is a regular value o
%    \end{itemize}
   % We say that $\chi \in \Theta(G)$ is good if it corresponds to a good cuspidal datum via Theorem Ber (8)??. Denote the set of these characters by $\cU_G$.

    $ $
    \begin{enumerate}
    
    \item

    By lemma \ref{lem:gen.irr} and the fact that $(W_{M,\rho,G})_{\rho}/M$ is finite we get that $\cU_G\subset \Theta(G)$ is an open dense set.
    %By a theorem of Bernstein theorem the map  (see \eqref{infchar}) is onto from the collection of all cuspidal datas for $G$ to $\Theta(G)$.
    %we define $\cU_G\subset \Theta(G)$ to be the image of the set of good cuspidal data under this map
    
    \item
    Let $\chi\in \cU_G$ and $(M,\rho)$ be \RamiI{a} corresponding cuspidal data. \Rami{To show that $\cM(G)|_{\chi}\cong Vect$, we will exhibit a projective generator for $\cM(G)|_{\chi}$ whose Endomorphism algebra is isomorphic to a matrix algebra.
    }
    %??$\pi^{-1}(\U_G)$ is the collection of good?? 
    %By Lemma ?? $\cU_G \Theta(G)$ is open dense

        By \Rami{Theorem \ref{thm:cent:dec} \eqref{Ber.cent:3}}\EitanF{,} the representation $\Psi_{G}(M,\rho)=i_M^G(ind_{\RamiM{M^0}}^M(\rho|_M))$ is a projective generator of $\cM_{[M,\rho]}$. Thus by Lemma \ref{ProjGenSpecialization} the representation $$\Rami{Q:=}(i_M^G(ind_{\RamiM{M^0}}^M(\rho|_{\RamiM{M^0}})))\otimes_{\fz(G)} \chi$$
    %?? check the equality?? this is wrong but it is isotypic by (7) of Bern?? 
	is a projective generator of the category $\cM(G)|_{\chi}$. 
	%\RamiQ{give a name to the projective generator, and explane the next aim}
	\Rami{In order to show that $\End_{G}(Q)$ is \RamiK{isomorphic to a} matrix algebra 
	we show that $Q$
%	will compute $Q$ and show that it is 
\RamiK{is isotypic  (semi-simple) representation}}.

	Theorem \Rami{\ref{thm:cent:blk} \eqref{thm:ber:ind}} gives us the morphism $\nu:\fz(G) \to \fz(M).$ 
	%and $\mu: \fz(M) \to C[M/\RamiM{M^0}]$. 
	Thus, via $\nu$, every representation $\tau \in \cM(M)$ have an action of $\fz(G)$. 
	%of $M$ we have an action of the  is allows us to
	We have 
	%\begin{multline*}
	
	\begin{align*}
        \Rami{Q} &\cong i_M^G(ind_{\RamiM{M^0}}^M(\rho|_{\RamiM{M^0}})\otimes_{\fz(G)} \chi)\\
        &\cong i_M^G(ind_{\RamiM{M^0}}^M(\rho|_{\RamiM{M^0}})\otimes_{\C[M/\RamiM{M^0}]}  (\C[M/\RamiM{M^0}] \otimes_{\fz(G)} \chi))
    \end{align*}

	%\end{multline*}
	{Our next step is to show that $\C[M/\RamiM{M^0}] \otimes_{\fz(G)} \chi$ is  a direct sum of characters.}
	
	\RamiI{Theorem \ref{thm:cent:blk} \eqref{thm:ber:azu0} gives us morphisms\RamiN{:} 
	\RamiM{$$\fz(M)\to Z(\cR_{(G,\rho)}) \cong O(\fX_G)^{\fI_\rho}\to O(\fX_G).$$
	Denote this composition by $\mu$. 
	
	{Since $\Theta(G)$ is a countable disconnected union of algebraic varieties, we can use the classical language of algebraic geometry when operating with it locally.}
	Since the action of $\fI_\rho$ on $\fX_G$ is free, we obtain that% 
	%$\mu: \fz(M) \to \C[M/\RamiM{M^0}]$
	}
	the corresponding map}  $$\mu_{*}: \Spec(\C[M/\RamiM{M^0}]) \to \Theta(M)$$ is \'{e}tale.
	
	Notice \RamiI{also} that $\nu$ induces a map $\nu_{*}: \Theta(M) \to  \Theta(G)$ that is \'{e}tale over $\chi$\RamiI{,} by \Rami{Theorem \ref{thm:cent:blk} \eqref{thm:ber:ind}} and the fact that $\chi \in \U_{G}$ is good. 
	
	\RamiH{Hence} the map $\nu_{*} \circ \mu_{*}: \Spec \C[M/\RamiM{M^0}] \to \Theta(G)$ is \'{e}tale at $\chi$. Let $\eta$ be the character of $\fz(M)$ acting on $\rho.$ We have that $\nu_{*}(\eta)=\chi.$ By \Rami{Theorem \ref{thm:cent:blk} \eqref{thm:ber:ind}}, $\nu_{*}^{-1}(\chi)=W_{M,\rho,G} \cdot \eta.$ 
	
	Therefore
	$$\C[M/\RamiM{M^0}] \otimes_{\fz(G)} \chi\cong \bigoplus_i \chi_i$$
	\Rami{with $\chi_{i} \in \fX_M$ and where
	%such that for any $i$ there is 
	%$w_i \in W_{M,\rho,G}$ 
	%such that 
	$$\chi_i \rho  \simeq Ad(w_i) \rho,$$ 
	for some $w_i \in W_{M,\rho,G}$\RamiI{.} 
	%where $\chi_{i} \in \fX_M.$
	}
	We get 
	$$\Rami{Q}\cong \bigoplus_i i_M^G(Ad(w_i) \rho).$$ By Lemma \ref{lem:JH}  all $i_M^G(Ad(w_i) \rho)$ are the same in the Grothendieck group \EitanD{of $\cM(G)$}. Since $\chi$ is good, \EitanD{these induced representations} are all irreducible and hence isomorphic. This implies that $\Rami{Q}$ is isotypic \RamiK{(semi-simple)}  \EitanE{representation} and hence  $$\cM(G)|_{\chi} \cong \cM(End_G(\Rami{Q})) \cong \cM(Mat_{k\times k}(\C))\cong Vect.$$
    \end{enumerate}
	
\end{proof}

\begin{cor}\label{cor:sph.are.all}
	Let \EitanD{$V=\Sc(G)_{U\times U,\psi\times \psi}$.}
	Let $\cU_G\subset \Theta(G)$ be as above.
	%There is  a Zariski open dense set $\cU_G\subset \Theta(G)$
	%??define this expresion??
	%s.t. for all $\fm \in \cU_G$ and for any irreducible $\rho\in \irr(G)$ that is annhilated by $\fm$, 
	\EitanD{For any $\chi\in \cU_G$, we let $\pi_{\chi} \in \inf^{-1}(\chi) \subset \Irr(G)$, be an irreducible representation with infinitesimal character $\chi.$} 
    \EitanD{Then} the space \EitanD{$(V^*)^{\fz(G),\chi^{-1}}$} is generated by the spherical characters of $\pi_{\chi}$.
\end{cor}
\begin{proof}
	By \cite[Proof of Proposition D]{AGS_Z}, for any $\chi\in \Theta(G)$, the space  $(V^*)^{\fz(G),\chi}$ is the space spanned by spherical characters of admissible representations of $G$ on which $\fz(G)$ acts by $\chi^{-1}$. The corollary follows now from the previous \EitanD{proposition} \RamiK{(Proposition \ref{cor:M.chi.vect}\eqref{cor:M.chi.vect:2})}.
\end{proof}

\subsection{\RamiK{Support of elements of Cohen-Macaulay modules}}\label{ssec:rel.tor}
\RamiK{In this subsection we show a strong restriction on the support of \EitanF{sections} of Cohen-Macaulay modules \EitanF{of full dimension}, see Corollary  \ref{cor:CMisRelTorFree}.}

\begin{definition}
Let $A$ be a commutative  unital finitely generated $\C$-algebra.
%We recall that an $A$-module $M$ is called of small torsion if
We say that an $A$-module $M$ is relatively torsion free if
for any non-zero element $m \in M$, \RamiK{and for any $x\in \Supp(m)$ we have  $\dim_x(\Supp(m))=\dim_x(M)$.}
\end{definition}

%\Rami{
%The following lemma is obvious
\begin{lemma}\label{lem: CMisRelTorFree}
Let $A$ be a commutative  unital finitely generated $\C$-algebra. %\RamiQ{globaly/locally}
\EitanE{An} \RamiK{equidimensional} Cohen-Macaulay module $M$ over $A$ is relatively torsion free.  
%\Rami{of full dimension is torsion free.} 
\end{lemma}
\begin{proof}
%\Rami{Without loss of generality, we may assume that $M$ is \RamiJ{equi-dimensional} Cohen-Macaulay \RamiJ{module}.}
%$\Spec(A)$ is irreducible?
\EitanE{By \cite[\RamiN{Criterion} 2.5]{BBG97} there exists a polynomial algebra} $B \subset A$ such that $M$ is finitely generated \RamiN{(locally)} free $B$-module.
Let $\pi: \Spec(A) \to \Spec(B)$ be the projection. Then $\pi|_{Supp(M)}$ is
a finite map. The support of $m$ in $\Spec(B)$ is the projection $\pi(\Supp(m))$ and the support of $M$ in $\Spec(B)$ is the projection $\pi(\Supp(M))$.
Now $$\dim(\Supp(m))=\dim(\Supp_B(m))=\dim(\Supp_B(M))=\dim(M).$$
%obtained by projecting
\end{proof}
%\RamiQ{add assumption of full dim}
\begin{cor}\label{cor:CMisRelTorFree}
\Rami{Let $A$ be a commutative unital finitely generated $\C$-algebra without zero divisors and let $M$ be a Cohen-Macaulay module of full dimension over $A$}. Let  $\RamiK{m} \in M$ {and} assume that there exist \RamiK{a dense subset} \Rami{$S \subset \Spec(A)$}
%$U \subset \Supp {(M)}$ 
such that for each $x \in S$ we have $\RamiK{m}|_x=0$. Then $\RamiK{m}=0$.
\end{cor}
\begin{proof}
Consider $M$ as a sheaf over  \Rami{$\Spec {(A)}$.} There exist an open dense subset  \RamiK{$U \subset \Spec(A)$}
%\RamiQ{why}
such that $M|_{\RamiK{U}}$ is locally free. 
%Since $U'$ is reduced 
The assumption implies that $\RamiK{m}|_{\RamiK{U}}=0$. 
By Lemma \ref{lem: CMisRelTorFree} we get $\RamiK{m}=0$.
%Thus previous lemma (Lemma \ref{lem: CMisRelTorFree}) implies that $f=0$.
\end{proof}

\subsection{\RamiK{Cohen-Macaulay property of  $\Sc(G)_{U\times U,\psi\times \psi}$ and the proof of Theorem \ref{thm:dens}}}

\begin{lemma}\label{lem:wit.CM}	
%add some assumptions ??e?
\RamiK{For any Bernstein block $\omega\in \Omega(G)$, t}he module 
    $\RamiK{\left(\Sc(G)_{U\times U,\psi\times \psi}\right)_\omega}$
	 %$ind_{U}^G(\psi)^R\otimes_{\cH(G)}ind_{U}^G(\psi)$ 
	 is {a \RamiK{finitely generated} Cohen-Macaulay module of full dimension} over \RamiK{$(\fz(G))_\omega.$}
\end{lemma}
%\RamiQ{locally f.g.}
%\RamiQ{add full dim}
\begin{proof}
\RamiK{The inversion on $G$ gives an anti-involution of $\cH(G)$. It allows to make a right module $V^R$ from a left module $V$ of $\cH(G)$. Note that 
$$\Sc(G)_{U\times U,\psi\times \psi}\cong ind_{U}^G(\psi)^R\otimes_{\cH(G)}ind_{U}^G(\psi).$$}

%Let $\Omega$ be the set of Bernstein blocks of $\cM(G)$. We can decompose 
%$$\cH(G)=\bigoplus_{\omega\in \Omega} \cH(G)_{\omega},$$ 
%	$$\fz(G)=\prod_{\omega\in \Omega} \fz(G)_{\omega},$$
	 %and
%$$ind_{U}^G(\psi)=\bigoplus_{\omega\in \Omega}
%ind_{U}^G(\psi)_{\omega}.$$ 
%So it is enough to show that for any component  $\omega\in \Omega$ the $\fz(G)_{\omega}$-module $(ind_{U}^G(\psi)_{\omega})^R \otimes_{\cH(G)_{\omega}} ind_{U}^G(\psi)_{\omega}$ is Cohen-Macaulay.
\RamiK{By} Theorem \ref{thm:proj} and Corollary \ref{cor:fg}
%This Lemma was erased: (4.1) Lemma \ref{lemma: fg}
%\cite[??]{AGS_Z}
%??should be a separate lemma that the module is f.g.??  the module 
the module $ind_{U}^G(\psi)_{\omega}$ is a direct summand of $\cH(G)_{\omega}^{n_\omega}$ for some $n_\omega$. So,  it is enough to check that 
$\cH(G)_{\omega}^{n_\omega}\otimes_{\cH(G)_{\omega}} \cH(G)_{\omega}^{n_\omega} \cong \cH(G)_{\omega}^{n_\omega^2}$ is \Rami{a Cohen-Macaulay module of full dimension} over $\fz(G)_{\omega}$.  This follows from Theorem \ref{thm:CM}. 
\end{proof}
%\RamiQ{Check of  full dim}
\begin{proof}[Proof of Theorem \ref{thm:dens}]
 	Let $C$ be the collection of characters of $\fz(G)$ corresponding to $\Pi$.
 	
 	%?? we need to deside whether we talk about characters or maximal idials???
 	
 Without loss of generality $C \subset \cU_G$ \RamiH{(see Definition \ref{def:good})}.
% 	Let \RamiQ{do not use $A$}
 	\RamiK{Let} {$V=\Sc(G)_{U\times U,\psi\times \psi}$} as above.
 	Fix $f\in V$\RamiH{.}
 	 We have to show that if  for any spherical character $\xi$ of an irreducible representation  $\pi\in\Pi$ we have $$\langle\xi,f\rangle=0,$$  then  $f=0.$ 
     By Corollary \ref{cor:sph.are.all} we obtain that $f|_\chi=0$ for any $ \chi \in C$.  Note that $V \cong ind_{U}^G(\psi)^R\otimes_{\cH(G)}ind_{U}^G(\psi)$. Thus, by Lemma \ref{lem:wit.CM}, the module $V$ is Cohen-Macaulay of full dimension over $\fz(G)$. {Passing to a single component of $\Theta(G)$ and applying Corollary \ref{cor:CMisRelTorFree} we obtain $f=0$.}
\end{proof}

%??? Spherical character needs 2 representations $\pi$  and $\tilde\pi$.  the $\fz$  character of $\tilde\pi$ is obtained by inversion from the the $\fz$  character of $\pi$.
%??? $(A/\fm A)=(ind_{U}^G(\psi)^R\otimes_{\cH(G)}ind_{U}^G(\psi))|_{\fm}=(ind_{U}^G(\psi)|_{inv(\fm)})^R\otimes_{\cH(G)|_{\fm}}ind_{U}^G(\psi)|_{\fm}$

\subsection{Proof of Lemma \ref{lem:JH}}\label{ssec:pf.JH}
%\begin{proof}
Let $I(\chi): i_M^G(\chi \rho) \to i_M^G(Ad(w)\chi \rho)$ be the intertwining operator defined for $Re(\chi)>>0$ (see e.g. \cite{Muic2008}).
By Lemma \ref{lem:gen.irr}, for generic $\chi$ it is an isomorphism. Fix now $f \in \cH(G)$, we have that 
$$tr(f,i_M^G(\chi \rho))=tr(f,i_M^G(Ad(w)\chi \rho))$$
for such $\chi$. But both sides are algebraic functions (e.g. Lemma 5.13 of \cite{Muic2008}) of $\chi$ hence, by linear independence of characters we obtain the result. 
%\end{proof}

\appendix
\section{\RamiN{Degenerate characters of  unipotent groups}}\label{app:deg}
\RamiN{In this appendix we prove some statement (Proposition \ref{prop:deg} below) on characters of maximal unipotent subgroups of $G$. We use this result in the next appendix in order to prove Theorem \ref{thm:BH}}

We will use the notations of \cite[\S 1]{BH03}. In particular 
we fix:
\begin{itemize}
    \item  $\bfS$ - a maximal split torus of $\bfG.$ 
    \item  $\bfT$ - the centralizer of $\bfS$ in $\bfG.$   
    \item A minimal parabolic $\bfB$ of $\bfG$, containing $\bfT$. 
    \item $\bfU$ - the unipotent radical of $\bfB.$
    \item $U:=\bfU(F)$.
    \item $\Phi$ - the set of relative roots corresponding to $(\bfG,\bfS)$.
     \item $\Phi^{+}$ - the set of positive roots corresponding to $\bfB$.   
    \item $\Delta$ - the set of simple roots corresponding to $\bfB$.

    %\EitanD{
    %\item $S:=\bfS(F)$.
    %\item For $s \in S$ and $\alpha \in \Delta,$
 %we denote by $s^{\alpha}$ the value of the character of $S$ correspodning to $\alpha$ at the point $s \in S.$    
    %}
\end{itemize}

%Our reduction \EitanD{to the non-degenerate case} is based on induction and the following:
\begin{prop}\label{prop:deg}
\RamiL{I}f $\xi$ is degenerate character of $U$ then there exist a proper parabolic $\bfP \subset \bfG$ containing $\bfB$ such that $\xi|_{\bfV(F)}=1$, where $\bfV\subset \bfP$ is the unipotent radical of $\bfP$
\end{prop}
%\EitanQ{I suggest to give the proof of the theorem now}.
For the proof we will need some recollections from \cite{Borel} and 
results from \cite{BH02}.
%\RamiQ{sujest to end the sentense here} concerning unipotent groups.
For any root $\alpha\in \Phi^+$ one can define a subgroup $U_{(\alpha)}\subset U$ as in \cite[Proposition 21.9 (i)]{Borel}  or \cite[section 1.1]{BH02}. 
%The following is proven in \cite{BH02}:
\begin{lemma}\label{lem:der}
If $\gamma\in \Phi^+$ is not colinear to a simple root, then $U_{(\gamma)}\subset [U,U]$
\end{lemma}
This lemma is proven in \cite{BH02}. As it is not formulated explicitly there, we indicate its proof.
%\begin{remark}
%If the characteristic of $F$ is not $2$, this is exactly \cite[Theorem 4.1 (a)]{BH02}. In case that the characteristic is $2$, and $G$ is absolutely almost simple, this follows from \cite[Theorem 2.1(c)]{BH02}. One can deduce the general case from the the absolutely almost simple case, in the same way that \cite[Theorem 4.1??]{BH02} is proven, namely using the study there of the effect of base change and isogeny on unipotent groups.
%\end{remark}

\begin{proof}[Proof of Lemma \ref{lem:der}]
Let $char(F)$ be the  characteristic of $F$. Recall that by \cite{BH02} we have $[U,U]=[\bfU,\bfU](F)$. \RamiL{We split the proof of the lemma into several cases}
\begin{enumerate}[{Case} 1:]
\item  $char(F) \ne 2$.
$ $ \\ 
%If the characteristic of $F$ is not $2$, 
In this case the statement follows from \cite[Theorem 4.1 (2)]{BH02}. 
\item  $char(F)=2$,  and $\bfG$ is absolutely almost simple.
$ $ \\ 
This case follows from \cite[Theorem 2.1(2,3)]{BH02}. 
\item  $char(F)=2$,  and $\bfG$ is $F$-almost simple, simply connected group.\\ By \cite[3.1.2]{Tit66} and \cite[Proposition 4.3]{BH02}, in this case $\bfG$ is a restriction of scalars of an absolutely almost simple group $\bfG'$.
Now, by 
%\cite[\S4.4 (Lemma) and \S4.5 (Lemma(2), Proposition)]{BH02} 
\cite[Lemma 4.4, Lemma 4.5(2) and Proposition 4.5]{BH02}
%the statement behaves well under restriction of scalars. Therefore
the result follows from the previous case.
\item 
$char(F)=2$ and $\bfG$ simply connected.\\
This case follows from the previous one, since any  simply connected group is a  product of \RamiL{simply connected,} almost simple groups.
\item $char(F)=2$,  and $G$ is semi-simple.\\
This case follows from the previous one, since by \cite[Proposition 5.2]{BH02} the  group $\bfU$ does not change when we replace $\bfG$ by an isogenus group.   
\item  $char(F)=2$.\\ 
\RamiL{This case follows from the previus one, since $\bfU\subset [\bfG,\bfG]$.}
\end{enumerate}
\end{proof}
One can assign to a set of simple roots $I \subset \Delta$ a maximal parabolic subgroup $\bfP_{I}$, see \cite[subsection 21.11, Proposition 21.12]{Borel}.
%with $I=\Delta-\{\alpha\}$. 
Let $\bfU_{I}$ be the unipotent radical of $\bfP_{I}$ and let $U_{I}=\bfU_{I}(F)$. Unfolding the concept of direct spanning  (see {\cite[subsection 14.3]{Borel}}) 
we have:
\begin{lemma}[cf. {\cite[Proposition 21.9 (ii)]{Borel}}]\label{lem:radicals}
For a subset $I \subset \Delta$ denote by $[I]$ the set of elements in $\Phi$ that are non negative integral combinations of elements of $I$. We have $$U_{I}=\prod_{\beta \in \Phi^+ \smallsetminus [I] }U_{(\beta)},$$
where the product can be taken in any order.
\end{lemma}
%\EitanQ{I think we can and should take only $\beta \in \psi_{nr}$, $\psi=\Phi^+ \smallsetminus [I]$
%}
\begin{proof}[Proof of
Proposition \ref{prop:deg}]
By the definition, a degenerate character of $U$ is a character that is trivial on $U_{(\alpha)}$ for some some $\alpha \in \Delta$. Let $\alpha$ be such that $\xi|_{U_{(\alpha)}}=1$ and let $\bfP:=\bfP_{\Delta-\{\alpha\}}$. It is left to show that $$\xi|_{U_{\Delta-\{\alpha\}}}=1.$$ 
By Lemma \ref{lem:radicals}, it is enough to show that for any $\beta \in \Phi^+ \smallsetminus [\RamiL{\Delta-\{\alpha\}}]$ we have $\xi|_{U_{(\beta)}}=1$.
As $\beta \in \Phi^+ \smallsetminus [\Delta-\{\alpha\}]$ we  will do it by considering two cases:
\begin{enumerate}[{Case} 1:]
    \item $\beta$ is not co-linear to a simple root.\\
    In this case the assertion follows from Lemma \ref{lem:der}
    \item $\beta$ is  co-linear to a simple root.\\
    In this case $\beta$ have to be co-linear to $\alpha$ and we have $U_{(\beta)}\subset U_{(\alpha)}$. By the assumption on $\alpha$ this implies that $\xi\RamiL{|}_{U_{(\beta)}}=1$.
\end{enumerate}
 \end{proof}
\section{\RamiN{Finite multiplicity result for degenerate Whittaker models}}\label{app:BH}

\RamiN{In this appendix we deduce Theorem \ref{thm:BH} from the special case where the character $\psi$ is non-degenerate, a result proven in {\cite[\S 4]{BH03}.}

The reduction \EitanD{to the non-degenerate case} is based on induction and Proposition \ref{prop:deg}.
}

\begin{proof}[Proof of Theorem \ref{thm:BH}]
The case when $\xi$ is non-degenerate is proven in \cite[{\S 4}]{BH03}.
We will prove the general case by induction on the dimension of $\bfG$. We can assume that  $\xi$ is degenerate.
Using Proposition \ref{prop:deg} we can find a \EitanF{proper} parabolic subgroup $\bfP\subset \bfG$ such that $\xi$  is trivial on $\bfV(F)$, where $\bfV\subset \bfP$ is the unipotent radical of $\bfP$. Let $\bfM$ be the (standard) Levi subgroup of $\bfP$. 

Note that $\xi$ is the pullback of a character $\xi_0$ of the group $\bfU(F)/\bfV(F)=(\bfU/\bfV)(F)=:\bfN(F)$.

We can identify $\bfN$ with the unipotent radical of the minimal parabolic subgroup of $\bfM$.

We have $$(\pi^*)^{(U,\xi)} \cong ((r_{M}^G(\pi))^*)^{(N,\xi_0)}.$$ 

Since $\pi$ is irreducible it is
%clearly 
finitely generated and admissible (see e.g. \cite[Theorem 12]{Ber_padic}). This implies that the representation $r_{M}^G(\pi)$ is finitely generated and admissible (see e.g. \cite[Proposition 19, Jacquet's lemma (page 64)]{Ber_padic})
%the representation $r_{M}^G(\pi)$ is finitely generated and admissible. 
Hence it is of finite length (See e.g. \cite[Theorem 6.3.10]{Cas}). The assertion follows now from the induction assumption.
\end{proof}

\bibliographystyle{alpha}
\bibliography{Ramibib}

\begin{thebibliography}{DKV84}

\bibitem[AAG12]{AAG}
Avraham Aizenbud, Nir Avni, and Dmitry Gourevitch.
\newblock Spherical pairs over close local fields.
\newblock {\em Comment. Math. Helv.}, 87(4):929--962, 2012.

\bibitem[AGS15]{AGS_Z}
Avraham Aizenbud, Dmitry Gourevitch, and Eitan Sayag.
\newblock {$\frak{z}$}-finite distributions on {$p$}-adic groups.
\newblock {\em Adv. Math.}, 285:1376--1414, 2015.

\bibitem[AS20]{AS_hom}
Avraham Aizenbud and Eitan Sayag.
\newblock Homological multiplicities in representation theory of {$p$}-adic
  groups.
\newblock {\em Math. Z.}, 294:451--469, 2020.

\bibitem[BBG97]{BBG97}
Joseph Bernstein, Alexander Braverman, and Dennis Gaitsgory.
\newblock The {C}ohen-{M}acaulay property of the category of
  {$(\mathfrak{g},K)$}-modules.
\newblock {\em Selecta Math. (N.S.)}, 3(3):303--314, 1997.

\bibitem[BBK18]{BBK}
Joseph Bernstein, Roman Bezrukavnikov, and David Kazhdan.
\newblock Deligne-{L}usztig duality and wonderful compactification.
\newblock {\em Selecta Math. (N.S.)}, 24(1):7--20, 2018.

\bibitem[Ber84]{Ber_cent}
J.~N. Bernstein.
\newblock Le ``centre'' de {B}ernstein.
\newblock pages 1--32, 1984.
\newblock Edited by P. Deligne.

\bibitem[Ber87]{Ber_sec}
Joseph Bernstein.
\newblock Second adjointness theorem for representations of p-adic groups.
\newblock 1987.

\bibitem[BH02]{BH02}
Colin~J. Bushnell and Guy Henniart.
\newblock On the derived subgroups of certain unipotent subgroups of reductive
  groups over infinite fields.
\newblock {\em Transform. Groups}, 7(3):211--230, 2002.

\bibitem[BH03]{BH03}
Colin~J. Bushnell and Guy Henniart.
\newblock Generalized {W}hittaker models and the {B}ernstein center.
\newblock {\em Amer. J. Math.}, 125(3):513--547, 2003.

\bibitem[Bor91]{Borel}
Armand Borel.
\newblock {\em Linear algebraic groups}, volume 126 of {\em Graduate Texts in
  Mathematics}.
\newblock Springer-Verlag, New York, second edition, 1991.

\bibitem[BR]{Ber_padic}
Joseph Bernstein and Karl Rumelhart.
\newblock Lectures on representations of p-adic groups.
\newblock
  \url{http://www.math.tau.ac.il/~bernstei/Unpublished_texts/Unpublished_list.html}.

\bibitem[Bus01]{Bus}
Colin~J. Bushnell.
\newblock Representations of reductive {$p$}-adic groups: localization of
  {H}ecke algebras and applications.
\newblock {\em J. London Math. Soc. (2)}, 63(2):364--386, 2001.

\bibitem[BZ76]{BZ}
I.~N. Bern\v{s}te\u{\i}n and A.~V. Zelevinski\u{\i}.
\newblock Representations of the group {$GL(n,F),$} where {$F$} is a local
  non-{A}rchimedean field.
\newblock {\em Uspehi Mat. Nauk}, 31(3(189)):5--70, 1976.

\bibitem[Cas]{Cas}
William Casselman.
\newblock Introduction to the theory of admissible representations of $p$-adic
  reductive groups.

\bibitem[CS19]{CS19}
Kei~Yuen Chan and Gordan Savin.
\newblock Bernstein-{Z}elevinsky derivatives: a {H}ecke algebra approach.
\newblock {\em Int. Math. Res. Not. IMRN}, (3):731--760, 2019.

\bibitem[DKV84]{DKV84}
P.~Deligne, D.~Kazhdan, and M.-F. Vign\'{e}ras.
\newblock Repr\'{e}sentations des alg\`ebres centrales simples {$p$}-adiques.
\newblock In {\em Representations of reductive groups over a local field},
  Travaux en Cours, pages 33--117. Hermann, Paris, 1984.

\bibitem[Kaz86]{Kaz86}
David Kazhdan.
\newblock Cuspidal geometry of {$p$}-adic groups.
\newblock {\em J. Analyse Math.}, 47:1--36, 1986.

\bibitem[LLS91]{LLS}
Ronnie Levy, Philippe Loustaunau, and Jay Shapiro.
\newblock The prime spectrum of an infinite product of copies of {${\bf Z}$}.
\newblock {\em Fund. Math.}, 138(3):155--164, 1991.

\bibitem[Mui08]{Muic2008}
Goran Mui\'{c}.
\newblock A geometric construction of intertwining operators for reductive
  {$p$}-adic groups.
\newblock {\em Manuscripta Math.}, 125(2):241--272, 2008.

\bibitem[Rot09]{Rot09}
Joseph~J. Rotman.
\newblock {\em An introduction to homological algebra}.
\newblock Universitext. Springer, New York, second edition, 2009.

\bibitem[Tit66]{Tit66}
J.~Tits.
\newblock Classification of algebraic semisimple groups.
\newblock In {\em Algebraic {G}roups and {D}iscontinuous {S}ubgroups ({P}roc.
  {S}ympos. {P}ure {M}ath., {B}oulder, {C}olo., 1965)}, pages 33--62. Amer.
  Math. Soc., Providence, R.I., 1966, 1966.

\end{thebibliography}
\end{document}